\documentclass[12pt,a4paper]{article}
\usepackage{preambule}

\renewcommand{\phi}{\varphi}
\newcommand{\sumtwo}[2]{\sum_{\substack{#1 \\ #2}}}
\newcommand{\gep}{\varepsilon}

\title{Scaling limits for the random walk penalized by its range in dimension one}
\author{Nicolas Bouchot}

\AtEndDocument{\bigskip{\footnotesize%
  \textsc{LPSM, Sorbonne Universit\'e, UMR 8001, Campus Pierre et Marie Curie, Bo\^ite courrier 158, 4 Place Jussieu, 75252 Paris Cedex 05, France.} \par  
  \textit{E-mail address}: \texttt{nicolas.bouchot@sorbonne-universite.fr} }}

\begin{document}

\pagestyle{plain}
\maketitle

\begin{abstract}
In this article we study a one dimensional model for a polymer in a poor solvent: the random walk on $\ZZ$ penalized by its range.
More precisely, we consider a Gibbs transformation of the law of the simple symmetric random walk by a weight $\exp(-h_n |\mathcal{R}_n|)$, with $|\mathcal{R}_n|$ the number of visited sites and $h_n$ a size-dependent positive parameter.
We use gambler's ruin estimates to obtain exact asymptotics for the partition function, that enables us to obtain a precise description of trajectories, in particular scaling limits for the center and the amplitude of the range. A phase transition for the fluctuations around an optimal amplitude is identified at $h_n \asymp n^{1/4}$, inherent to the underlying lattice structure. \newline

\end{abstract}

\section{Introduction of the model and main results}

Consider a simple symmetric random walk $(S_k)_{k \geq 0}$ on $\ZZ^d$, $d\geq 1$, starting from $0$, with law denoted $\PP$. For $h > 0$, we define the following Gibbs transformation of $\PP$, called the \emph{polymer measure}
\[ \dd \PP_{n,h}(S) = \frac{1}{Z_{n,h}} e^{-h|\mathcal{R}_n(S)|} \dd \PP(S), \]
where $\mathcal{R}_n(S) \defeq \big\{ S_0, \dots, S_n \big\}$ is the range of the random walk up to time $n$ and $|\cdot|$ is the cardinal measure. The normalizing quantity
\[ Z_{n,h} = \esp{e^{-h|\mathcal{R}_n(S)|}} \]
is called the \textit{partition function} and is such that $\PP_{n,h}$ is a probability measure on the space of trajectories of length $n$.

\medskip
In any dimension $d \geq 1$, the asymptotics for the log-partition function are known since Donsker and Varadhan \cite{DV-Partequiv-79}. These asymptotics strongly suggest that a polymer of length will typically fold in (and fill up) a ball of radius $\rho\, n^{\frac{1}{d+2}}$ for some specific constant $\rho = \rho(d,h)$. This has been proved by \cite{bolthausen1994localization} in dimension $d=2$, but only much more recently in dimension $d\geq 3$, by Berestycki and Cerf \cite{Berestycki} and Ding, Fukushima, Sun and Xu \cite{ding2020geometry}. More precisely, for $h=1$ (easily generalized to any $h>0$), they prove that there exists a positive $\rho_d$, which only depends on the dimension $d$, such that for any $\eps > 0$,
\[ \lim_{n\to\infty} \probaM{n,1}{\exists x \in \RR^d, B\big( x,(1-\eps)\rho_d n^{\frac{1}{d+2}} \big) \cap \ZZ^d \subset \mathcal{R}_n \subset B\big( x,(1+\eps)\rho_dn^{\frac{1}{d+2}} \big)} = 1 \,, \]
where $B(x,r)$ is the $d$-dimensional Euclidean ball centered at $x$ with radius $r$.

\smallskip
In dimension $d=1$, this is much easier since the range is uniquely determined by its two endpoints (and always fills completely the one-dimensional ball). This allows for more explicit calculations using mostly gambler's ruin estimates. In particular, one easily derives that $n^{-1/3}|\mathcal{R}_n|$ converges to $\big(\frac{\pi^2}{h}\big)^{1/3}$ in $\PP_{n,h}$-probability.

\subsection{Outline of the paper}

In the rest of the work, we focus only on the case of dimension $d=1$. Also, we allow the penalization intensity to depend on the length of the polymer, meaning $h = h_n$ now depends on $n$.
We exploit gambler's ruin estimates to their full potential and derive exact asymptotics for the partition function (not only for the log-partition function). Afterwards, we will be able to prove a scaling limit (actually we prove a local limit theorem) for the joint law of the center $W_n$ and the amplitude $T_n$ of the range:
\[ T_n \defeq \max_{k \leq n} S_k - \min_{k \leq n} S_k = |\mathcal{R}_n| - 1, \qquad W_n \defeq \frac{T_n}{2} + \min_{k \leq n} S_k = \frac{1}{2} \left( \max_{k \leq n} S_k + \min_{k \leq n} S_k \right). \]

For the sake of the exposition, let us consider the case
\begin{equation}
\label{def:hn}
\lim_{n\to\infty} n^{-\gamma} h_n = \hat h \in (0,+\infty)   \,,\qquad \text{ for some } \gamma\in \RR\,.
\end{equation}
Some results are already presented in \cite{berger2020one} which considers a disordered version of the model:
\begin{itemize}
    \item[(i)] if $\gamma < -\frac{1}{2}$ then $\PP_{n,h_n}$ is close in total variation to $\PP$;
    \item[(ii)] if $\gamma \in (-\frac{1}{2},1)$ then $(\frac{n \pi^2}{h_n})^{-1/3} T_n$ converges to $1$ in $\PP_{n,h_n}$-probability.
    \item[(iii)] if $\gamma > 1$ then $\PP_{n,h_n}$ is concentrated on trajectories visiting only two sites.
\end{itemize}
Since cases (i) and (iii) are degenerate, we focus on the case $\gamma \in (-\frac{1}{2},1)$. In this paper, we give another proof of the convergence $(\frac{n \pi^2}{h_n})^{-1/3}  T_n \to 1$ and we additionally identify the fluctuations of $T_n-(\frac{n \pi^2}{h_n})^{1/3}$.
We find that a phase transition occurs at $\gamma = \frac14$ for the fluctuations: 
\begin{itemize}
    \item[(i)] if $\gamma < \frac14$ then the fluctuations, normalized by $(\frac{n}{h_n^4})^{1/6}$, converge to a Gaussian variable;
    \item[(ii)] if $\gamma > \frac14$ then the range penalization is strong enough to collapse the range on $(\frac{n\pi^2}{h_n})^{1/3}$ in the sense that the fluctuations live on a finite set (of cardinality $1$, sometimes $2$).
\end{itemize}
\par We will also prove that $(\frac{n\pi^2}{h_n})^{-1/3} W_n$ converges to a random variable with density $\frac{\pi}{2}\cos(\pi u) \mathbbm{1}_{[-\frac12,\frac12]}(u)$ with respect to the Lebesgue measure and is independent of the fluctuations. This type of results appears to be folklore for confined polymers, as the density is the eigenfunction associated with the principal Dirichlet eigenvalue of the Laplacian on $[0,1]$ (see e.g.\ \cite[Ch.~8]{den2009random}), but we are not aware of a proof written in detail (at least for the random walk penalized by its range).

\begin{notation}
In the rest of the paper we shall use the standard notations:
as $x\to a$, we write
$g(x)\sim f(x)$ if $\lim_{x\to a} \frac{g(x)}{f(x)} =1$,
$g(x) = \bar{o}(f(x))$ if $\lim_{x\to a} \frac{g(x)}{f(x)} = 0$,
$g(x) = \grdO(f(x))$ if  $\limsup_{x\to a}\big| \frac{g(x)}{f(x)} \big| < +\infty$ and $f \asymp g$ if $g(x) = \grdO(f(x))$ and $f(x) = \grdO(g(x))$.

\smallskip
We also extensively use the following notation: for $\mathcal A$ an event, we denote
\[
Z_{n,h_n} (\mathcal A) \defeq \esp{e^{-h_n|\mathcal{R}_n(S)|} \indic{S\in \mathcal A}} \,,
\]
so that in particular $\PP_{n,h_n}(\mathcal A) = \frac{1}{Z_{n,h_n}} Z_{n,h_n} (\mathcal A)$.
\end{notation}

\subsection{Main results}

The following two theorems summarize our results, the first being the main result regarding the asymptotic behavior of $(T_n,W_n)$ and the second being asymptotics for $Z_{n,h_n}$ that have a use of their own.

We define the following quantities, that will be used throughout the paper:
\begin{equation}
\label{def:Tnan}
T_n^* = T_n^*(h_n) \defeq \left( \frac{n \pi^2}{h_n}\right)^{1/3},
\qquad 
a_n = a_n(h_n) \defeq \frac{1}{\sqrt{3}} \left( \frac{n \pi^2}{h_n^4}\right)^{1/6} =\frac{1}{\sqrt{3n\pi^2}}(T_n^*)^2 .
\end{equation}
Note that  $\lim\limits_{n\to\infty} a_n =+\infty$ if and only if $\lim\limits_{n\to\infty} n^{-1/4} h_n = 0$.

\begin{theorem}
\label{th-limite}
\textbullet\ Assume that $h_n \geq n^{-1/2} (\log n)^{3/2}$ and $\lim\limits_{n\to\infty} n^{-1/4} h_n = 0$; in other words, $\gamma \in (-\frac12,\frac14)$ in~\eqref{def:hn}. Then under $\PP_{n,h_n}$, we have the following convergence in distribution
\[ 
\left( \frac{T_n - T_n^*}{a_n} \, ; \, \frac{W_n}{T_n^*} \right) \xrightarrow[n \to +\infty]{(d)} (\mathcal T, \mathcal W), 
\]
where the random variables $\mathcal T$ and $\mathcal W$ are independent with $\mathcal T \sim \mathcal N(0,1)$ and $\mathcal W$ with density given by $\frac{\pi}{2} \cos(\pi u) \mathbbm{1}_{[-\frac12,\frac12]}(u)$.

\smallskip
\textbullet\ Assume that $\lim\limits_{n\to\infty} n^{-1/4}  h_n =+\infty$ and $\lim\limits_{n\to\infty} n^{-1} h_n = 0$; in other words, $\gamma \in (\frac14,1)$ in~\eqref{def:hn}. Denote $t_n^o$ the decimal part of $T_n^* - 2$ and define $\mathcal{A}_n$ as $\mathset{0}$ if $t_n^o < \frac12$, $\mathset{1}$ if $t_n^o > \frac12$ and $\mathset{0,1}$ if $t_n^o = \frac12$. Then we have 
\[
\lim_{n\to\infty}\PP_{n,h_n} \big( T_n - \lfloor T_n^*-2 \rfloor \not\in \mathcal{A}_n \big) = 0.
\]
Also, under $\PP_{n,h_n}$ we have the convergence in distribution $\frac{W_n}{T_n^*}\xrightarrow{(d)}\mathcal W$.
\end{theorem}

\begin{remark}
The term $a_n = \frac{1}{\sqrt{3n\pi^2}}(T_n^*)^2$ in Theorem~\ref{th-limite} arises naturally as a Taylor expansion coefficient in the exponential part of the partition function after injecting gambler's ruin formulae, see Section~\ref{sec:heuristics} below.
\par The assumption that $h_n \geq n^{-1/2} (\log n)^{3/2}$ is due to technicalities in the proof of Theorem \ref{prop equiv proba} below and gambler's ruin formulae.
\end{remark}

\begin{theorem}
\label{th-equiv}
We have the following exact asymptotics:

\textbullet\  Assume that $h_n \geq n^{-1/2} (\log n)^{3/2}$ and $\lim\limits_{n\to\infty} n^{-1/4} h_n = 0$; in other words, $\gamma \in(-\frac12,\frac14)$ in~\eqref{def:hn}. Then, as $n\to\infty$, 
\[ Z_{n,h_n} =(1+\bar{o}(1))  \frac{16\sqrt{2}}{\sqrt{3\pi}} \Big(\frac{\cosh(h_n)-1}{h_n}\Big) \sqrt{n}  \,\exp\Big( -\frac32 h_n T_n^* \Big)\,. \]

\smallskip
\textbullet\  Assume that $\lim\limits_{n\to\infty} n^{-1/4}  h_n =+\infty$ and $\lim\limits_{n\to\infty} n^{-1} h_n = 0$; in other words, $\gamma \in (\frac14,1)$ in~\eqref{def:hn}. Denote by $t_n^o$ the decimal part of $T_n^* - 2$. Then, as $n \to \infty$,
\[ Z_{n,h_n} = \frac{16}{\pi^{4/3}} (1 + \indic{t_n^o = \frac12}) \Big(\frac{\cosh(h_n)-1}{h_n^{1/3}}\Big) n^{1/3}  \,\exp\Big( -\frac32 h_n T_n^* - \Phi_n(\indic{t_n^o \geq \frac12 + \frac{1}{T_n^o}}) \frac{\pi n}{(T_n^*)^4} (1+\bar{o}(1)) \Big)\,. \]
with $\Phi_n(t) \defeq 6 + \frac{\pi^2}{12} + \frac32 \Big[\frac{1}{T_n^o}|t - t_n^o| + (t - t_n^o)^2 \Big]$. Note that $\lim\limits_{n\to\infty} \frac{n}{(T_n^*)^4} = +\infty$ means $\bar o(\frac{n}{(T_n^*)^4})$ could still diverge.
\end{theorem}

For the sake of completeness, we add the following result concerning the critical case $\lim\limits_{n\to\infty} n^{-1/4}  h_n =\hat h \in (0,+\infty)$. We write $\varsigma_n(t) \defeq \frac{1}{T_n^o}|t- t_n^o| \indic{t \in \mathset{0,1}} + (t- t_n^o)^2$

\begin{proposal}\label{prop-1/4}
Suppose that $\lim\limits_{n\to\infty} n^{-1/4}  h_n = \hat{h} \in (0,+\infty)$, so in particular we have $\lim_{n\to\infty} a_n = \frac{\pi^{1/3}}{\sqrt{3}\hat{h}^{2/3}} \eqdef a$. Then, as $n\to\infty$, we have
\[ Z_{n,h_n} =(1+\bar{o}(1)) \frac{16}{\pi^{4/3}} \Big(\frac{\cosh(h_n)-1}{h_n^{1/3}}\Big) n^{1/3}  e^{-\frac32 h_n T_n^*} \theta_n(a), \quad \text{with $\theta_n(a) \defeq \sum\limits_{t = -\infty}^{+\infty} e^{-\frac{\varsigma_n(t)}{2a^2}}$}\]
 (recall that $\delta_n = \indic{t_n^o \geq \frac12} - t_n^o$ with $t_n^o$ the decimal part of~$T_n^* - 2$).
Furthermore, for any integers $r \leq s$, as $n\to\infty$ we have
\[ \PP_{n,h_n} \left( r \leq T_n - \lfloor T_n^*-2 \rfloor \leq s \right) =(1+\bar{o}(1)) \frac{1}{\theta_n(a)} \sum_{t = r}^{s} e^{-\frac{\varsigma_n(t)}{2a^2}} \,. \]
\end{proposal}

\subsection{Range's endpoints and confinement estimates}

\par Let us now state some estimates for the probability
that the range of a random walk is exactly a given interval.
The proof is postponed to Section~\ref{sec-ruine} 
and follows from gambler's ruin estimates that can be found in \cite[Chap. XIV]{Feller}.

Let $x,y$ be two non-negative integers and denote by $E_x^y(n)$ the following event
\[
E_x^y(n) \defeq \left\{ \mathcal R_n = \llbracket -x,y\rrbracket \right\} = \left\{ M_n^-  = -x \, , \, M_n^+ = y \right\} \,,
\]
where we also introduced $M_n^- \defeq \min_{k \leq n} S_k$,
$M_n^+ \defeq \max_{k \leq n} S_k$,
and used the standard notation $\llbracket a,b\rrbracket = [a,b]\cap \mathbb Z$.
We also define the following function $g$,
that encodes the exponential decay rate of confinement probabilities
inside a strip: 
\begin{equation}
\label{g DL}
    g(T)  \defeq - \log \cos\left( \frac{\pi}{T} \right) = \frac{\pi^2}{2T^2} + \frac{\pi^4}{12T^4} + \grdO(T^{-6}) \quad \text{ as }T\to\infty.
\end{equation} 
The main result used in the rest of the paper is the following. It is based on sharp gambler's ruin estimates, see Lemmas~\ref{lemme-ruine}-\ref{lemme conf strict} in Section \ref{sec-ruine}.

\begin{theorem}\label{prop equiv proba}
We have the following convergence for any positive $T = T(n)$
\begin{equation}\label{eq-equiv-proba}
    \varlimsup_{n \to \infty} \sup_{\substack{x,y \in \NN\\ x+y=T}} \left| \frac{\proba{E_x^y(n)}}{\Theta_n(x,y)} - 1 \right| = 0
\end{equation}
Where we defined the function $\Theta_n(x,y)$ for $x+y = T$ as
 \[ \Theta_n(x,y) \defeq \begin{dcases}
    \frac{4}{\pi} \sin \left( \frac{\pi (x+1) }{T} \right)  e^{-g(T+2)n} & \text{ if } \frac{n}{T^3} \to +\infty\\
    \frac{4}{\pi}  (e^{\alpha \pi^2}-1)\left[ e^{\alpha \pi^2}\sin \left( \frac{\pi(x+1)}{T} \right)  -  \sin \left( \frac{\pi x}{T} \right) \right] e^{-g(T)n} & \text{ if } \frac{n}{T^3} \to \alpha \in (0,+\infty) \\
    \frac{4 \pi^3n^2}{T^6} \left[ \sin\left( \frac{\pi x}{T} \right)  + \frac{T^2}{\pi n}  \right] e^{-g(T)n}  & \text{ if } \frac{n}{T^3} \to 0
   \end{dcases} \]
\end{theorem}

\begin{remark}
For the rest of the paper we will prefer to write $\proba{E_x^y(n)} = (1+\bar{o}(1)) \Theta_n(x,y)$ with $\bar{o}(1)$ uniform in $x,y$ and only depending on $T=x+y$, in the sense of \eqref{eq-equiv-proba}.
\end{remark} 

We can summarize these results in a more compact way,
if we agree to lose some precision in the case where $x$ is close to $0$ (more precisely if $x/T\to 0$):
\begin{equation}
\label{eq:probasimple}
\proba{E_x^y(n)} = \psi\Big( \frac{n \pi^2}{T^3}\Big)  \left[ \sin\left( \frac{\pi x}{T} \right) + \bar{o}(1) \right] e^{-g(T+2)n}, \quad \text{ with } \psi(\alpha) \defeq \frac{4}{\pi} ( 1-  e^{- \alpha})^2.
\end{equation}
Here, the $\bar{o}(1)$ is uniform in $x,y$ and depends only on $T=x+y$.
To get~\eqref{eq:probasimple}, we have used in particular
that $g(T+2)n-g(T)n \sim \frac{2\pi^2 n}{T^3}$ as $T\to\infty$,
which converges to $2\alpha \pi^2$ if $\lim_{n\to\infty} \frac{n}{T^3} = \alpha \in [0,+\infty)$.

\begin{remark}
Whenever $x=0$ (or $y=0$ using symmetry) we have the same theorem applied to $x = 0$ (see Section \ref{Exy positive RW}), except when $\frac{n}{T^3} \to 0$ in which case
\[ \proba{E_0^T(n)} =(1+\bar{o}(1)) \frac{4n\pi}{T^3}  \sin \left( \frac{\pi}{T+2} \right) e^{-g(T+1)n} \,. \]
This will not be significant starting from Section \ref{sec-part} as it only consists of two not-so-peculiar range configurations among the \textit{many} configurations in the partition function.
\end{remark}

Let us stress that one easily deduces from Theorem \ref{prop equiv proba} the following statement, leading to the asymptotic independence in Theorem~\ref{th-limite}, as well as the convergence in distribution of~$\frac{W_n}{T_n^*}$ to $\mathcal{W}$.

\begin{proposal}\label{prop:cv-loi}
Let $(t_n)_{n\geq 1}$ be any sequence of integers such that $\lim_{n\to\infty}t_n=\infty$ and $\frac14 t_n^2 \log t_n \leq n$.
Then, conditioning on $T_n = t_n$,
$\frac{W_n}{t_n}$ converges in distribution to~$\mathcal{W}$. 
More precisely, we have the following local limit convergence: uniformly for $w$ such that $2w \in \llbracket -t_n, t_n\rrbracket$,
\[
\proba{W_n = w \, | \, T_n = t_n} = \frac{\pi}{2} \left[ \cos \left( \frac{w\pi}{t_n} \right) +\bar{o}(1) \right] \,\quad \text{as }n\to\infty.
\]     
\end{proposal}

\noindent
Note that this proposition allows us to focus our study on $T_n$ instead of $(T_n,W_n)$.

\begin{proof}
For $-\frac12 \leq a \leq b \leq \frac12$, we get thanks to~\eqref{eq:probasimple} that
\begin{multline*}
    \proba{a \leq \frac{W_n}{t_n} \leq b \, ; T_n = t_n} = \sum_{\substack{x+y=t_n\\ 2 at_n \leq y-x \leq 2 b t_n}} \proba{E_x^y(n)} \\
     = \frac{4}{\pi} \Big( 1-e^{-\frac{n\pi^2}{t_n^3}}\Big)^2 e^{-g(t_n+2)n} \sum_{2at_n \leq 2w \leq 2bt_n} \left[ \cos \left( \frac{w\pi}{t_n} \right) + \bar{o}(1)\right] \,,
\end{multline*}
where we have set $w = w(x,y) \defeq \frac{y-x}{2}$.
Similarly,
\[ \proba{T_n = t_n} =\frac{4}{\pi} \Big(1-e^{-\frac{n\pi^2}{t_n^3}}\Big)^2 e^{-g(t_n+2)n} \sum_{-t_n\leq 2w \leq  t_n} \left[ \cos \left( \frac{w\pi}{t_n} \right) + \bar{o}(1)\right] \,. \]
We therefore end up with
\[
    \proba{a \leq \frac{W_n}{t_n} \leq b \, \Big| \, T_n = t_n} = \frac{\sum_{at_n \leq w \leq bt_n} \left[ \cos \left( \frac{w\pi}{t_n} \right) + \bar{o}(1)\right]}{\sum_{-t_n \leq 2w \leq t_n} \left[ \cos \left( \frac{w\pi}{t_n} \right) + \bar{o}(1)\right]} \xrightarrow[n\to \infty]{} \frac{\pi}{2} \int_a^b \cos(\pi u) \, \dd u \,,
 \]
and taking $a=b=w/t_n$,
\[\proba{W_n = w \, | \, T_n = t_n} = \frac{\cos \left( \frac{w\pi}{t_n} \right) + \bar{o}(1)}{\sum_{-t_n \leq 2w \leq t_n} \left[ \cos \left( \frac{w\pi}{t_n} \right) + \bar{o}(1)\right]} =  \frac{\pi}{2} \left[ \cos \left( \frac{w\pi}{t_n} \right) +\bar{o}(1) \right]\,. \qedhere\]
\end{proof}

\subsection{Some heuristics}
\label{sec:heuristics}

Let us present some heuristics for obtaining the asymptotics of the partition function,
and explain how the quantities $T_n^*$ and $a_n$ (recall~\eqref{def:Tnan}) appear.
We can decompose the partition function as
\[
Z_{n,h_n}  = \sum_{x,y\geq 0} e^{-h_n(T+1)} \proba{E_x^y(n)}  \, ,
\]
where we have set $T=T(x,y) = x+y$.
In view of Theorem~\ref{prop equiv proba},
we have $\proba{E_x^y(n)} = u_n(x,y) e^{ -g(T)n}$
with  $g(T) = (1+\bar{o}(1)) \frac{\pi^2}{2T^2}$.
Hence, the main contribution to the sum will come from $x,y$
with $T$ that is close to minimizing 
the function
\begin{equation}\label{eq:def-phin}
\phi_n(T) \defeq h_n T + \frac{n \pi^2}{2T^2} \,.
\end{equation}
Then, notice that $\phi_n$ is minimal at $T=  T_n^{*} \defeq \big( \frac{n \pi^2}{h_n}\big)^{1/3}$ (recall~\eqref{def:Tnan})
and that  $\phi_n(T_n^{*}) = \frac{3 \pi^{1/3}}{2} n^{1/3} h_n^{2/3} = \frac32 h_n T_n^{*}$.

Let us now factorize $e^{\phi_n(T_n^{*})}$ (and $e^{h_n}$) in the sum above, to get that
\[
 e^{\frac32 h_n T_n^{*}} e^{h_n} Z_{n,h_n} 
 \approx  \sum_{x,y\geq 0} u_n(x,y) \exp\big( - (\phi_n(T)-\phi_n(T_n^{*}) ) \big)  \,.
\]
Now, since $\phi_n'(T_n^*) =0$,
we have
$\phi_n(T) \approx \phi_n(T_n^{*}) + (T-T_n^{*})^2 \phi_n''(T_n^{*})$,
with $\phi_n''(T_n^{*}) =  \frac{3n \pi^2}{(T_n^{*})^4} = \frac{1}{a_n^2}$ (recall~\eqref{def:Tnan}).
In the sum above the main contribution therefore comes from 
values of~$T$ that are such that $\phi_n(T)-\phi_n(T_n^{*})$ is at most of order $1$,
that is with $T-T_n^{*} = \grdO( a_n)$.

\smallskip
Let us stress once more that if $\lim\limits_{n\to\infty} n^{-1/4} h_n = 0$ then $\lim\limits_{n\to\infty} a_n = +\infty$, whereas if $\lim\limits_{n\to\infty} n^{-1/4} h_n = +\infty$ then $\lim\limits_{n\to\infty} a_n = 0$.
The condition $h_n \geq n^{-1/2} (\log n)^{3/2}$ ensures that $\frac14 (T_n^*)^2 \log T_n^* \leq n$ and the condition $\lim\limits_{n\to\infty} n^{-1}h_n =0$ ensures that $\lim\limits_{n\to\infty} T_n^* = +\infty$.

\subsection{Further comments on the results}

Theorem \ref{th-limite} states that asymptotically, the polymer behaves as a random walk whose range's size $T_n$ fluctuates around the optimal $T_n^* = \big( \frac{n \pi^2}{h_n}\big)^{1/3}$.
If $h_n n^{-1/4} \to 0$ (weak penalization), then the fluctuations are Gaussian at a scale $a_n =\frac{1}{\sqrt{3}} \big(\frac{n}{h_n^4} \big)^{1/6}$.
On the other hand, if $h_n n^{-1/4} \to \infty$ (strong penalization), then the fluctuations vanish and $T_n$ is equal to either $\lfloor T_n^* \rfloor - 2$ or $\lfloor T_n^*\rfloor -1$.
In both cases, the relative position of the center of the range is asymptotically independent of its size, with distribution given by the density $\frac{\pi}{2} \cos(\pi u) \mathbbm{1}_{[-\frac12,\frac12]}(u)$, conjectured or discussed in previous works (see \cite[Theorem 8.3]{den2009random} for example) but with no concrete proof (to the best of our knowledge).

\subsubsection{Continuous analogue of the model}

One can easily see the similarities between this polymer model and the study of the Brownian motion penalized by the amplitude of its trajectory.
For a Brownian motion $\beta$, define $|C_T| \defeq |\mathset{\beta_t \, : \, t \leq T}|$ its amplitude at time $T$ (here $|\cdot|$ is the Lebesgue measure). Donsker and Varadhan proved in \cite{saucisse} that
\[ \lim_{T \to \infty} \frac{1}{T^{1/3}} \log \esp{e^{-\nu |C_T|}} = -\frac{3}{2} (\nu \pi)^{2/3} \,. \]
\par Schmock later expanded on this result in \cite{schmock1990} and obtained that the associated Gibbs measures $\PP_{T,\nu}(\dd \omega) = e^{-\nu |C_T|}\mathbb{W}(\dd \omega)$  (with $\mathbb{W}$ the Wiener measure) converge weakly to a measure $\PP_{\infty, \nu}$ given by
\[ \PP_{\infty,\nu} (A) = \int_0^{c_\nu} \frac{\pi}{2 c_\nu} \sin \left( \frac{\pi u}{c_\nu} \right) P_{u-c_\nu, u}(A) \, \dd u \,, \]
with $c_\nu = (\pi^2/\nu)^{1/3}$, where $P_{u-c_\nu, u}$ denotes the path measure of a Brownian taboo process with taboo set $\mathset{u-c_\nu, u}$.
In other words, $\PP_{\infty, \nu}$ is a mixture of taboo processes $P_{u-c_\nu, u}$, which correspond to the actual diffusion process conditioned to stay in an interval of length $c_{\nu}$ and upper edge $u$; additionally, the mixing measure selecting the upper edge $u$ is identical to $\mathcal{W}$ in Theorem~\ref{th-limite} (if one selects the center of the range rather than the upper edge). This is therefore completely analogous to our Theorem~\ref{th-limite}.


However,  because there is no underlying lattice, the continuous case should not display a transition for the fluctuations at $\nu=\nu_T \asymp T^{1/4}$: when $\lim_{T\to\infty} T^{-1/4}\nu_T = +\infty$, fluctuations become $\bar{o}(1)$ but still remain Gaussian after a proper scaling.
Let us also stress that in the continuous case, well-known results such as Lévy triple law (see \cite[Theorem 6.18]{schilling2021brownian}) allow for relatively simple computations of the law of the endpoint $\beta_T$ for a large $T$ conditioning on the range's endpoints --- which Theorem~\ref{prop equiv proba} does not provide in our setting, we only get the position of the starting point relative to the range, see Proposition~\ref{prop:cv-loi}. Obtaining a result for the starting and endpoint for our model would require the joint law of $(M_n^-, M_n^+, S_n)$ or a study based on local times of the polymer, which are both beyond the scope of this paper.

\subsubsection{Other related models}

Related models for self-interacting polymers have been studied in the literature these past years. We mention here two of these models and their recent advancements.

\medskip
First, one can consider a disordered version of the random walk penalized by its range, \textit{i.e.}\ the case where the penalization by the range is perturbed by a random environment. Take a collection of i.i.d variables $(\omega_z)_{z \in \ZZ}$ and consider the random polymer measure
\[ \dd \PP_{n,h}^{\omega,\beta}(S) = \frac{1}{Z_{n,h}^{\beta,\omega}} \exp \Big( \sum_{z \in \mathcal{R}_n(S)} \big(\beta \omega_z - h\big) \Big) \dd \PP(S), \]
in particular $\PP_{n,h} = \PP_{n,h}^{\omega,0}$. This quenched model was studied in~\cite{berger2020one, berger2021non, huang2019scaling}, for size-dependent parameters $h_n$ and $\beta_n$. 
In dimension $d=1$, \cite{berger2020one} finds a wide range of behaviors for the polymer depending on the sign and the growth speed of the parameters $h_n,\beta_n$. However, several questions remain open, such as determining the location and fluctuations of the range (in the spirit of Theorem \ref{th-limite}) in a regime where the range size (properly rescaled) converges to a non-random quantity --- we are currently investigating this question~\cite{bouchot2022}.

\medskip
Another related model is the charged polymer, where charges are attached to the different monomers and interact with each other, see \cite[Chapter 8]{den2009random} for an overview.
Take i.i.d.\ random variables $(\omega_k)_{k \in \NN}$, and consider the following quenched Gibbs measure on random walk trajectories
\[ \dd \PP_{n,\beta}^{\omega}(S) = \frac{1}{Z_{n,\beta}^{\omega}} \exp \Big( - \beta \sum_{1 \leq i < j \leq n} \omega_i \omega_j \indic{S_i = S_j} \Big) \dd \PP(S) \,. \]
\par Some recent papers~\cite{berger:hal-01576410,caravenna2016annealed,10.1007/978-981-15-0302-3_1} are dealing with the annealed version of the model, that can be written in the following form
\[ \dd \PP_{n,\beta}^{\mathrm{ann}}(S) = \frac{1}{Z_{n,\beta}^{\mathrm{ann}}} \exp \Big( - \sum_{x\in \ZZ^d} g_{\beta}(\ell_n(x)) \Big) \dd \PP(S) , \]
where $\ell_n(x) = \sum_{i=1}^n\indic{S_i=x}$ is the local time at site $x$ and where $g_{\beta}$ is a function that depends on $\beta$ and on the distribution of $\omega$.
This model has been shown to undergo a folding/unfolding phase transition, and the case of dimension $d=1$ has been investigated in remarkable detail in~\cite{caravenna2016annealed}.
Our model falls in the same class of models: it corresponds to using the function $h\indic{\ell_n(x) > 0}$ instead of the function $g_{\beta}(\ell_n(x))$; note that our model also displays a folding/unfolding transition when $h$ goes from positive to negative values.

\subsubsection*{Organization of the rest of the paper}

The rest of the paper is organized as follows:
\begin{itemize}
    \item In Section~\ref{sec:weak} we focus on the case of a ``weak'' penalization, that is $\lim\limits_{n\to\infty} n^{-1/4} h_n =0$: we give local asymptotic estimates for the partition function (Lemma~\ref{lem:local}), from which we deduce the first point of both Theorem~\ref{th-equiv} and Theorem~\ref{th-limite} (in that order).
    \item In Section~\ref{sec:strong} we treat the case of a ``strong'' penalization, that is $\liminf\limits_{n\to\infty} n^{-1/4} h_n >0$: we modify the arguments of Section~\ref{sec:weak} to provide local asymptotic estimates for the partition function (Lemma~\ref{lem:local}). From this, we deduce first the second point of Theorems~\ref{th-limite}-\ref{th-equiv}, \textit{i.e.}\ in the case $\lim_{n\to\infty}n^{-1/4} h_n=+\infty$, before we turn to the border case of Proposition~\ref{prop-1/4}, \textit{i.e.}\ $\lim\limits n^{-1/4}h_n =\hat h \in (0,+\infty)$.
    \item Finally, in Section~\ref{sec-ruine} we derive sharp gambler's ruin estimates (see Lemmas~\ref{lemme-ruine}-\ref{lemme conf strict}) and their consequences for the range of a random walk, that is we prove Theorem~\ref{prop equiv proba}.
\end{itemize}

\section{Weak penalization: the case \texorpdfstring{$\lim\limits_{n\to\infty} n^{-1/4}h_n=0$}{}}
\label{sec:weak}

\subsection{Local asymptotics for the partition function}

Our first preliminary result computes 
the contribution of the partition function 
from trajectories with a fixed size of the range $T_n$,
with $T_n= T_n^{*} +\bar{o}(T_n^{*})$.
Recall that
 $T_n^{*} \defeq \big( \frac{n \pi^2}{h_n}\big)^{1/3}$
and $a_n \defeq \frac{1}{\sqrt{3n\pi^2}} (T_n^{*})^2$.

\begin{lemma}
\label{lem:local}
Assume that $h_n \geq n^{-1/2} (\log n)^{3/2}$ 
and that $\lim_{n\to\infty} n^{-1/4} h_n=0$.
Let $(\gep_n)_{n\geq 1}$ be any vanishing sequence.
Then, for any $t\in \mathbb Z$ and $w\in \frac12\mathbb Z$ such that $|t|\leq \varepsilon_n T_n^*$ and $w -\frac12(\lfloor T_n^* \rfloor +t) \in 2 \mathbb Z$, we have
\begin{equation}
\label{part equiv fin}
    Z_{n,h_n}\Big( T_n = \lfloor T_n^* \rfloor + t  , W_n = w \Big) =  \psi_n  \times  \Big(\cos\Big( \frac{\pi w}{T_n^*} \Big) +\bar{o}(1) \Big) \times e^{- (1+\bar{o}(1)) \frac{ t^2 }{2 a_n^2}   } \,,
\end{equation}
where the $\bar{o}(1)$ only depends on $\gep_n$,
and where we have set:
\begin{equation}
\label{def:psin}
\psi_n \defeq \psi(h_n) \exp\Big( -h_n(T_n^*+1) - g(T_n^*+2) n  \Big) \,,
\quad \text{ with }\ \psi(\alpha) = \frac{4}{\pi} ( 1-  e^{- \alpha})^2\,.
\end{equation}
\end{lemma}

\begin{proof}
We have
\[
Z_{n,h_n}\Big(T_n = \lfloor T_n^* \rfloor + t  , W_n = w   \Big)  = e^{-h_n(x+y+1)} \proba{E_x^y(n)}  \, ,
\]
with $x+y = \lfloor T_n^* \rfloor + t$ and
$\frac12(y-x) =  w$.
Thanks to~Theorem~\ref{prop equiv proba},
we can estimate this term.
Indeed,
for every $x,y$
such that $\lim_{n\to\infty} \frac{x+y}{T_n^*}=1$,
using~\eqref{eq:probasimple}
we have
\begin{equation}
\label{eq:local1}
\proba{E_x^y(n)}
= \psi\Big(\frac{n \pi^2 }{(T_n^*)^3} \Big) \,\Big(\sin\Big( \frac{\pi x}{T_n^*} \Big) +\bar{o}(1) \Big) e^{ - g(x+y+2) n}   \,,
\end{equation}
where the $\bar{o}(1)$ is uniform in $x,y$
and  $\psi(\alpha) = \frac{4}{\pi} ( 1-  e^{- \alpha})^2$.
Note that by the definition of $T_n^*$ we have $\frac{n \pi^2}{(T_n^*)^3} = h_n$.

Now, here we have that $x= \frac12 (\lfloor T_n^* \rfloor + t) -w$,
with $\frac{1}{2T_n^* } (\lfloor T_n^* \rfloor + t) \to \frac12$.
Hence, we can write 
$\sin( \frac{\pi x}{T_n^*}) = \cos(\frac{\pi w}{T_n^*} ) +\bar{o}(1)$ in~\eqref{eq:local1}.
Recall the definition \eqref{eq:def-phin} $\phi_n(T) = h_n T + \frac{n \pi^2}{2T^2}$ and write $h_n(T+1)+g(T+2) n = \varphi_n(T) + h_n+\tilde g(T) n$,
with $\tilde g(T) =g(T+2) -\frac{\pi^2}{2T^2}$,
to get that
\begin{equation}
\label{eq:localZ}
\begin{split}
Z_{n,h_n}&\Big(T_n = \lfloor T_n^* \rfloor + t  , W_n = w   \Big) \\
& = \psi(h_n) \,\Big(\cos\Big( \frac{\pi w}{T_n^*} \Big) + \bar{o}(1) \Big) \exp\Big( -\phi_n(\lfloor T_n^* \rfloor + t) - h_n-\tilde g(\lfloor T_n^* \rfloor + t) n \Big)\,.
\end{split}
\end{equation}

We can use that
$\phi_n'(T_n^*) =0$,
$\phi_n''(T_n^*) = \frac{3n \pi^2}{(T_n^{*})^4} = 1/a_n^2$
and $\phi_n'''(T) = - \frac{12 n \pi^2}{T^5}$
to get that
 for any $\frac12  T_n^* \leq T\leq 2 T_n^*$
\[
\bigg|\varphi_n(T) - \varphi_n(T_n^*) - \frac{(T-T_n^*)^2 }{2 a_n^2} \bigg| \leq C \frac{|T-T_n^*|^3 n}{(T_n^*)^5} = C' \frac{|T-T_n^*|^3}{T_n^*\, a_n^2} \,.
\]
Hence, using that $a_n\to\infty$,
we get that 
\begin{equation}
\label{phiapprox}
\phi_n(\lfloor T_n^* \rfloor + t) = \varphi_n(T_n^*) + (1+\bar{o}(1)) \frac{t^2}{2 a_n^2} \,.
\end{equation}

We can also perform the same expansion
for $\tilde g(T) = g(T+2)-\frac{\pi^2}{2T^2}$,
for which $\tilde g'(T) = \frac{\pi^2}{T^3} - \frac{\pi}{(T+2)^2} \tan(\frac{\pi}{T+2})$:
\[
|\tilde g(T) - \tilde g(T_n^*) | \leq  C \frac{ |T-T_n^*|}{ (T_n^*)^4}\,,
\] 
so that, inserting $T = \lfloor T_n^* \rfloor + t$
\begin{equation}
\label{tildegapprox}
|\tilde g(\lfloor T_n^* \rfloor + t)  n- \tilde g(T_n^*) n| \leq C \frac{|t+\delta_n|}{a_n^2} = \bar{o}(1) \frac{ t^2}{a_n^2} + \bar{o}(1)\end{equation}
as $n\to\infty$, uniformly in $t$ (consider separately the case $|t|\leq a_n$ and $|t|\geq a_n$).
 
All together, plugging~\eqref{phiapprox}-\eqref{tildegapprox}
into~\eqref{eq:localZ}, 
we end up with
the desired result, with
$\psi_n \defeq \psi(h_n) \exp( -h_n- \phi_n(T_n^*) -\tilde g(T_n^*) n )$ which coincides with the definition~\eqref{def:psin} above.
\end{proof}

\subsection{Asymptotics of the partition function}
\label{sec-part}

Lemma~\ref{lem:local} allows us to obtain the correct 
behavior for the partition function.

\begin{proof}[Proof of Theorem~\ref{th-equiv}]
Assume that $h_n\geq n^{-1/2} (\log n)^{3/2}$
and that $\lim_{n\to\infty} n^{-1/4}h_n=0$, so in particular $a_n\to+\infty$.

Note that by \cite{berger2020one} (or simply using large deviation principles),
we have for any $\gep>0$ 
\[
\lim_{n\to\infty} \PP_{n,h_n}\Big( | T_n- T_n^*| >\gep T_n^* \Big) =0 \,.
\]
Therefore, one can find some vanishing sequence $(\gep_n)_{n\geq 0}$
such that we have the asymptotic equivalence
$Z_{n,h_n} = (1+\bar{o}(1)) Z_{n,h_n} ( | T_n- T_n^*| \leq \gep_n T_n^*)$.
We therefore only have to estimate that last partition function.
We may decompose it as
\[
\begin{split}
 Z_{n,h_n}&\Big(  |T_n - T_n^*| \leq \gep_n T_n^* \Big) 
  = \sum_{ t = -  \lfloor \gep_n T_n^*\rfloor }^{\lfloor \gep_n T_n^* \rfloor} \sum_{ -(T_n^* +t)  \leq 2w \leq T_n^* +t  } Z_{n,h_n}\Big( T_n = \lfloor T_n^*\rfloor +t  \,, W_n =w \Big) \,.
 \end{split}
\]
Therefore, thanks to Lemma~\ref{lem:local},
using that $\frac{|t|}{a_n^2} = \bar{o}(1) \frac{t^2}{a_n^2} +\bar{o}(1)$ 
as $n\to\infty$ with a $\bar{o}(1)$ uniform in~$t$ (since $a_n\to\infty$),
we get that 
\[
\begin{split}
 Z_{n,h_n}\Big(   |T_n - T_n^*| \leq \gep_n T_n^*  \Big) 
 & = (1+\bar{o}(1))  \psi_n \sum_{ t =  -  \lfloor \gep_n T_n^*\rfloor }^{ \lfloor \gep_n T_n^* \rfloor}  
  e^{- (1+\bar{o}(1)) \frac{ t^2 }{2 a_n^2} }
  \sumtwo{ -(T_n^* +t)  \leq 2w \leq T_n^* +t  }{ w \in \frac12( \lfloor T_n^* \rfloor  +t) + \mathbb Z } 
  \cos\Big( \frac{\pi w}{T_n^*} \Big)  \,.
  \end{split}
\]
\par Now, as $T_n^*$ goes to $+\infty$,
the internal sum is a Riemann sum:
we have, uniformly for $ (1-\gep_n)T_n^*  \leq t \leq  (1+\gep_n) T_n^* $,
\[
    \sumtwo{ -(T_n^* +t)  \leq 2w \leq T_n^* +t  }{ w \in \frac12( \lfloor T_n^* \rfloor  +t) + \mathbb Z } 
  \cos\Big( \frac{\pi w}{T_n^*} \Big)
  \sim  T_n^* \int_{-\frac12}^{\frac12} \cos(\pi v) \, \dd v = \frac{2}{\pi}  T_n^*\,.
\]
Then, a sum over $t$ remains, which is also a Riemann sum:
as $a_n \to+\infty$ and $\gep_n T_n^*/a_n \to +\infty$ ($T_n^*/a_n \geq (cst.) \sqrt{\log n}$ so such sequence $(\eps_n)$ exists), we have
\[
\sum_{ t =  -  \lfloor \gep_n T_n^*\rfloor }^{ \lfloor \gep_n T_n^* \rfloor}  
  e^{- (1+\bar{o}(1)) \frac{ t^2 }{2 a_n^2} }
   \sim a_n  \int_{-\infty}^{\infty}  e^{-  \frac{ u^2 }{2} } \dd u 
    = \sqrt{2\pi}  a_n \,.
 \]

All together,
we have proved that, as $n\to\infty$
\begin{equation}
\label{eq:equivZnh}
Z_{n,h_n} \sim \frac{2\sqrt{2}}{\sqrt{\pi}} \psi_n a_n T_n^*\,.
\end{equation}
Recalling the definition of $T_n^*$ and $a_n$,
we have $a_n T_n^* = \frac{\pi}{\sqrt{3}} \frac{\sqrt{n}}{h_n}$.
Additionally,
recalling the definition~\eqref{def:psin} of $\psi_n$,
and using that $g(T+2) = \frac{\pi^2}{2T^2} - \frac{2\pi^2}{T^3} + \grdO(\frac{1}{T^4})$ as $T\to\infty$,
we get that
\[
\psi_n = \psi(h_n) \exp\Big( -h_n T_n^*-h_n - \frac{n \pi^2}{2 (T_n^*)^2} + \frac{2\pi^2 n}{ (T_n^*)^3} +\bar{o}(1) \Big) \,,
\]
since $\lim_{n\to\infty} \frac{n}{(T_n^*)^4} =0$ because 
$\lim_{n\to\infty} n^{-1/4}h_n=0$.
By the definition of $T_n^* $ we have $\frac{n \pi^2}{ (T_n^*)^3} = h_n$,
we get that $\psi_n \sim \psi(h_n) e^{h_n} e^{- \frac32 h_n T_n^*}$.
Putting all estimates together and noting that $e^{\alpha} (1-e^{-\alpha})^2= 2(\cosh(\alpha)-1)$,
this concludes the proof.
\end{proof}

\begin{proof}[Proof of Theorem~\ref{th-limite}]
The proof reduces to showing the following Lemma.
\begin{lemma}
\label{lem:restricted}
Let $h_n\geq n^{-1/2} (\log n)^{3/2}$
be such that $\lim\limits_{n\to\infty} n^{-1/4}h_n=0$.
Then, for any $r<s$,
we have
\[
 \lim_{n\to\infty} \frac{1}{\psi_n a_n T_n^*} Z_{n,h_n}\Big(  r\leq \frac{ |T_n - T_n^*|}{a_n} \leq s \Big) =  \frac{2}{\pi} \int_r^s e^{- \frac{u^2}{2}} \dd u \,,
\]
where $\psi_n$ is the sequence that appears in Lemma~\ref{lem:local}.
\end{lemma}

Indeed, once we have this lemma, in view of the asymptotics~\eqref{eq:equivZnh} and Proposition \ref{prop:cv-loi},
we get that for any $r<s$ and any $a<b$,
\[
\begin{split}
\PP_{n,h_n}\Big(  r\leq \frac{ |T_n - T_n^*|}{a_n} \leq s, a \leq  \frac{W_n}{T_n^*} \leq b \Big) &
= \frac{1}{Z_{n,h_n}} Z_{n,h_n}\Big(  r\leq \frac{ |T_n - T_n^*|}{a_n} \leq s, a \leq  \frac{W_n}{T_n^*} \leq b \Big)\\
&\xrightarrow{n\to\infty}  \int_r^s \frac{1}{\sqrt{2\pi}} e^{-\frac{u^2}{2}} \dd u  \int_{a}^{b} \frac{\pi}{2}\cos(\pi v) \mathbbm{1}_{[-\frac12,\frac12]} \dd v \,,
\end{split}
\]
which concludes the proof.
\end{proof}

\begin{proof}[Proof of Lemma~\ref{lem:restricted}]
The proof proceeds as for the proof of Theorem~\ref{th-equiv}.
We can decompose the partition function
as
\[
\begin{split}
\sum_{ t = \lfloor r a_n \rfloor }^{ \lfloor s a_n \rfloor} & \sumtwo{ -\lfloor T_n^* \rfloor  -t  \leq 2w \leq \lfloor T_n^* \rfloor + t}{ w \in \frac12( \lfloor T_n^* \rfloor  +t) + \mathbb Z }  Z_{n,h_n}\Big( T_n = \lfloor T_n^*\rfloor +t  \,, W_n =w \Big) \\
& = (1+\bar{o}(1)) \psi_n \sum_{ t = \lfloor r a_n \rfloor }^{ \lfloor s a_n \rfloor}  
  e^{- (1+\bar{o}(1))\frac{ t^2 }{2 a_n^2} }
  \sumtwo{ -\lfloor T_n^* \rfloor  -t  \leq 2w \leq \lfloor T_n^* \rfloor + t}{ w \in \frac12( \lfloor T_n^* \rfloor  +t) + \mathbb Z } 
  \cos\Big( \frac{\pi w}{T_n^*} \Big)  \,,
  \end{split}
\]
where we have used~Lemma~\ref{lem:local} as above (using that $a_n\to\infty$).

Again, as $T_n^*$ goes to $+\infty$, the internal sum is a Riemann sum:
we have, uniformly for $ \lfloor T_n^* \rfloor + \lfloor r a_n \rfloor  \leq t \leq  \lfloor T_n^* \rfloor + \lfloor s a_n \rfloor $,
\[
  \sumtwo{ -\lfloor T_n^* \rfloor  -t \leq 2w \leq \lfloor T_n^* \rfloor + t}{ w \in \frac12( \lfloor T_n^* \rfloor  +t) + \mathbb Z } 
  \cos\Big( \frac{\pi w}{T_n^*} \Big)
  \sim  T_n^* \int_{-\frac12}^{\frac12} \cos(\pi v) \, \dd v \sim \frac{2}{\pi} T_n^* \,.
\]
Then, the sum over $t$ that remains is also a Riemann sum: as $a_n \to+\infty$, we have
\[
\sum_{ t =  \lfloor s a_n \rfloor }^{\lfloor t a_n \rfloor}  
  e^{- (1+\bar{o}(1)) \frac{ t^2 }{2 a_n^2} }
  \sim  a_n  \int_s^t  e^{-\frac{ u^2 }{2} } \dd u \,,
 \]
which concludes the proof.
\end{proof}

\section{Strong penalization and vanishing fluctuations}
\label{sec:strong}

\subsection{The case \texorpdfstring{$\liminf\limits_{n\to\infty} n^{-1/4}h_n = +\infty$}{}}

See that the case where $a_n \to 0$ is much more restrictive to establish an analog of Lemma~\ref{lem:local}, as $\grdO(a_n^{-2})$ quantities now bring extremely large contributions to the exponential part of $Z_{n,h_n}$ and slight deviations from the optimal size $T_n^*$ will be penalized by a large factor. Indeed, if we are to get a \textit{continuity} from Theorem \ref{th-limite} when $\limsup_{n\to\infty} a_n < \infty$, we want to know the exact asymptotic law of fluctuations without renormalization. We denote 
\begin{equation}
    \label{def:barphi}
    \bar{\phi}_n(T) = h_n(T+1) + \frac{n\pi^2}{2(T+2)^2} 
\end{equation}
and $T_n^o \defeq\argmin \bar{\phi}_n(T)$ the (more exact) optimal amplitude of the range.

\begin{lemma}\label{lem:local0}
Assume that $\lim_{n \to \infty} n^{-1/4}h_n = +\infty$ and $\lim_{n \to \infty} n^{-1}h_n = 0$ and let $(\eps_n)_{n\geq 1}$ be any vanishing sequence. Then, for any $t \in \ZZ \setminus \mathset{0,1}$ and $w \in \frac12 \ZZ$ such that $|t| \leq \eps_n T_n^o$ and $w -\frac12(\lfloor T_n^o \rfloor + t) \in 2 \mathbb Z$, we have
\begin{equation}
\label{part equiv fin 0}
    Z_{n,h_n}\Big( T_n = \lfloor T_n^* -2\rfloor + t  , W_n = w \Big) =  \bar{\psi}_n  \times  \Big(\cos\Big( \frac{\pi w}{T_n^*} \Big) +\bar{o}(1) \Big) \times e^{- (1+\bar{o}(1)) \frac{ (t- t_n^o)^2 }{2 a_n^2}} \,,
\end{equation}
where $t_n^o \defeq T_n^o - \lfloor T_n^o \rfloor$, $\bar{o}(1)$ is a vanishing quantity that depends only on $\eps_n$ and
\begin{equation}
\label{def:psin0}
\bar{\psi}_n \defeq \psi(h_n) \exp\Big( -h_n(T_n^*-1) - g(T_n^*) n  \Big) \,,
\quad \text{ with }\ \psi(\alpha) = \frac{4}{\pi} ( 1-  e^{- \alpha})^2\,.
\end{equation}
When $t \in \mathset{0,1}$ we instead have
\begin{equation}
\label{part equiv fin 0-1}
    Z_{n,h_n}\Big( T_n = \lfloor T_n^* -2\rfloor + t  , W_n = w \Big) =  \bar{\psi}_n  \times  \Big(\cos\Big( \frac{\pi w}{T_n^*} \Big) +\bar{o}(1) \Big) \times e^{- (1+\bar{o}(1)) \frac{1}{2 a_n^2} \left[ \frac{1}{T_n^o}|t- t_n^o| + (t- t_n^o)^2 \right]} .
\end{equation}
\end{lemma}

\begin{proof}
We can perform the same decomposition as in Lemma~\ref{lem:local} and setting $T=\lfloor T_n^* -2 \rfloor +t$, we arrive at (analogously to~\eqref{eq:localZ})
\begin{equation}
\label{eq:localZ2}
\begin{split}
Z_{n,h_n}\Big( T_n = T  , W_n = w \Big) & = \psi (h_n)    \Big(\cos\Big( \frac{\pi w}{T} \Big) +\bar{o}(1) \Big) \times e^{- h_n(T+1) - g(T+2) n}\\
& = \psi(h_n) \Big(\cos\Big( \frac{\pi w}{T} \Big) +\bar{o}(1) \Big) \times e^{- \bar \phi_n(T) - \bar g(T) n}
\end{split}
\end{equation}
with $\bar \phi_n$ defined above in~\eqref{def:barphi} and where $\bar g(T) \defeq \frac{\pi^2}{T+2} - g(T+2)$.
One can then easily check that $T_n^o = T_n^* - 2$ and that 
$\bar{\phi}_n''(T_n^o) = -\frac{3n\pi^2}{(T_n^o + 2)^4}$ and $\bar{\phi}_n^{(3)}(T_n^o) = \frac{12n\pi^2}{(T_n^o + 2)^5}$,
thus
\[ \left| \bar{\phi}_n(T) - \bar{\phi}_n(T_n^o) - \frac{3n\pi^2}{2(T_n^o + 2)^4}(T-T_n^o)^2 \right| \leq \frac{12n\pi^2}{(T_n^o + 2)^5}(T-T_n^o)^3 = \bar{o}(1) \frac{(T-T_n^o)^3}{a_n^2}. \]
Furthermore, the first two orders of the Taylor expansion of $\bar{g}(T)$ around $T_n^o$ are given by
\begin{multline*}
   \bar{g}(T_n^o) + (T-T_n^o) \left( \frac{\pi^2}{(T_n^o + 2)^3} - \frac{\pi}{(T_n^o + 2)^2} \tan \frac{\pi}{T_n^o + 2} \right)\\
   + \frac{(T-T_n^o)^2}{2}  \left( -\frac{3\pi^2}{(T_n^o+2)^4} + \frac{2\pi}{(T_n^o+2)^3} \tan \frac{\pi}{T_n^o + 2} + \frac{\pi^2}{(T_n^o+2)^4} \Big(1+\tan^2 \frac{\pi}{T_n^o + 2} \Big) \right) \,.
\end{multline*}
A Taylor expansion of the tangent leads to the following bound
\begin{multline*} 
n|\bar{g}(T) - \bar{g}(T_n^o)| \\
\leq |T-T_n^o|\frac{n\pi^3}{3(T_n^o + 2)^5} + (T-T_n^o)^2 \frac{n\pi^3}{3(T_n^o + 2)^6} = \frac{1+\bar{o}(1)}{a_n^2 T_n^o} \left[|T-T_n^o| + (T-T_n^o)^2 \right] \,.
\end{multline*}
In particular, we get the lemma injecting $T = \lfloor T_n^* -2 \rfloor +t = T_n^o - t_n^o + t$
in~\eqref{eq:localZ2}, using also that $\bar \phi_n(T_n^o) + n \bar g(T_n^o)= h_n(T_n^*-1) + n g(T^*_n)$ and $|T-T_n^o| > 1$ for $t \not\in \mathset{0,1}$.
\end{proof}

\begin{proof}[Proof of Theorem \ref{th-equiv}]

Suppose $\liminf\limits_{n\to\infty} n^{-1/4}h_n = +\infty$ then $a_n \to 0$ and we can't apply Riemann summations as in the proof of Theorem \ref{th-equiv} or Lemma \ref{lem:restricted}. However Lemma \ref{lem:local0} allows to exclude the slightest deviations using $a_n^{-2} \to \infty$ that we use to estimate
\[ \begin{split}
 Z_{n,h_n}\Big(   |T_n - T_n^*| \leq \eps_n T_n^*  \Big) 
 & = (1+\bar{o}(1))  \psi_n \sum_{ t =  -  \lfloor \gep_n T_n^*\rfloor }^{ \lfloor \gep_n T_n^* \rfloor}  
e^{- (1+\bar{o}(1)) \frac{\varsigma_n(t)}{2 a_n^2} }
  \!\!\! \sumtwo{ -(T_n^* +t)  \leq 2w \leq T_n^* +t  }{ w \in \frac12( \lfloor T_n^* \rfloor  +t) + \mathbb Z } 
  \!\!\!  \cos\Big( \frac{\pi w}{T_n^*} \Big) \, ,
  \end{split} \]
where we recall wrote $\varsigma_n(t) \defeq \frac{1}{T_n^o}|t- t_n^o| \indic{t \in \mathset{0,1}} + (t- t_n^o)^2$.
As $T_n^*$ goes to $+\infty$, the internal sum still is a Riemann sum and thus, uniformly for $ (1-\gep_n)T_n^o  \leq t \leq  (1+\gep_n) T_n^o $,
\[
    \sumtwo{ -(T_n^* +t)  \leq 2w \leq T_n^* +t  }{ w \in \frac12( \lfloor T_n^* \rfloor  +t) + \mathbb Z } 
  \cos\Big( \frac{\pi w}{T_n^*} \Big)
  \sim  T_n^* \int_{-\frac12}^{\frac12} \cos(\pi v) \, \dd v = \frac{2}{\pi}  T_n^*\,.
\]
\par Now, for the  sum that remains, we factorize by the largest term, attained at $t = 0$ or $t = 1$ depending on $t_n^o$. Thus denote $\delta_n^2 = \varsigma_n(t)$ with $t = 0$ or $1$ such that $\delta_n^2 = \varsigma_n(0) \wedge \varsigma_n(1)$, meaning $\delta_n^2 = \varsigma_n(\indic{t_n^o \geq \frac12 + \frac{1}{T_n^o}})$. We have
\begin{equation}\label{Znh decomp 0}
    \sum_{ t =  -  \lfloor \eps_n T_n^*\rfloor }^{ \lfloor \eps_n T_n^* \rfloor}  
  e^{- (1+\bar{o}(1))\varsigma_n(t)} = e^{-(1+\bar o(1)) \delta_n^2} \Bigg( 1 + e^{\frac{1}{2 a_n^2}(\delta_n^2 -  \varsigma_n(1-\indic{t_n^o \geq \frac12 + \frac{1}{T_n^o}}))} + \sum_{\substack{t =  -  \lfloor \gep_n T_n^*\rfloor\\ t \neq 0,1}}^{ \lfloor \gep_n T_n^* \rfloor} e^{\delta_n^2 - \frac{(t-t_n^o)^2}{2a_n^2}} \Bigg) \,.
\end{equation} 
First use dominated convergence to get
\[ 0 \leq \sum_{\substack{t =  -  \lfloor \gep_n T_n^*\rfloor\\ t \neq 0,1}}^{ \lfloor \gep_n T_n^* \rfloor} e^{\frac{\delta_n^2 - (t-t_n^o)^2}{2a_n^2}} \leq \sum_{\substack{t \in \ZZ\\ t \neq 0,1}}
  e^{-\frac{1}{2a_n^2}(t-1)(t+1-t_n^o)} \xrightarrow[n \to \infty]{} 0 \,. \]
Then, note that when $t_n^o \neq \frac12$, the second term of \eqref{Znh decomp 0} goes to $0$, whereas when $t_n^o = \frac12$ it is equal to $1$. Thus we have
\[ Z_{n,h_n} = \frac{2}{\pi} \big( 1 + \indic{t_n^o = \frac12} \big) \psi(h_n) e^{h_n - \frac32 h_n T_n^* - \frac{\delta_n^2}{2a_n^2} + \bar{o}(a_n^{-2})} \exp \left( g(T_n^*+2)n - \frac{n\pi^2}{2T_n^*} - \frac{2}{(T_n^*)^3} \right) \,, \]
where we recall the definition $\psi(a) = \frac{4}{\pi} ( 1-  e^{- a})^2$. Note that $g(T_n^*+2)n - \frac{n\pi^2}{2T_n^*} - \frac{2}{(T_n^*)^3} = \grdO \big( \frac{n}{(T_n^*)^4} \big) = \grdO(a_n^{-2})$ now goes to infinity.
We could push the asymptotic expansion to any arbitrary order, but since we already have the term $\bar{o}(a_n^{-2}) = \bar{o}(\frac{n}{(T_n^*)^4})$, we simply use the following expansion up to order $4$:
\[ g(T_n^*+2)n - \frac{n\pi^2}{2T_n^*} - \frac{2}{(T_n^*)^3} = (72 + \pi^2) \frac{n\pi^2}{12(T_n^*)^4} (1 + \bar{o}(1)) \,. \]
Using $e^{\alpha} (1-e^{-\alpha})^2= 2(\cosh(\alpha)-1)$ and writing $a_n^{-2}$ explicitly concludes the proof.
\end{proof}

\begin{proof}[Proof of Theorem \ref{th-limite}]
The above proof of Theorem \ref{th-equiv} already shows that
\[ \frac{1}{Z_{n,h_n}} Z_{n,h_n}\Big( |T_n - \lfloor T_n^*-2 \rfloor| \in \mathset{0,1} \Big) \xrightarrow[n\to\infty]{} 1  \,. \]
According to \eqref{Znh decomp 0}, the proof only consists of finding which values of $t$ contribute to $e^{-\frac{\delta_n^2}{a_n^2}}$. By the definition of $\delta_n$ we easily see that if $t_n^o < \frac12$ it is $t = 0$, if $t_n^o > \frac12$ it is $t=1$, and when $t_n^o = \frac12$ both have the same contribution. This observation leads to defining $\mathcal{A}_n$ as announced in Theorem \ref{th-limite}, and Proposition \ref{prop:cv-loi} completes the proof.
\end{proof}

\subsection{Case \texorpdfstring{$\lim\limits_{n \to \infty} n^{-1/4} h_n = \hat{h}$}{}: order one fluctuations}

\begin{proof}[Proof of Proposition \ref{prop-1/4}]
Going back to use Lemma \ref{lem:local0}, we get for any vanishing sequence $(\gep_n)_{n\geq 1}$
\[ 
 Z_{n,h_n}\Big(   |T_n - T_n^*| \leq \eps_n T_n^*  \Big) 
  = (1+\bar{o}(1))  \bar \psi_n \sum_{ t =  -  \lfloor \gep_n T_n^*\rfloor +2 }^{ \lfloor \gep_n T_n^* \rfloor+2}  
e^{- (1+\bar{o}(1)) \frac{\varsigma_n(t)}{2 a_n^2} }
 \!\!\! \sumtwo{ -(T_n^* +t)  \leq 2w \leq T_n^* +t  }{ w \in \frac12( \lfloor T_n^* \rfloor  +t) + \mathbb Z }   \!\!\!
  \cos\Big( \frac{\pi w}{T_n^*} \Big)  \,.
\]
\par The internal Riemann sum is dealt with the same method as before, while we can take the limit for $a_n$ in the external sum. Thus we have, as $n\to\infty$
\[ \begin{split}
 Z_{n,h_n}\Big(   |T_n - T_n^*| \leq \eps_n T_n^*  \Big) \sim \frac{2}{\pi} T_n^* \bar \psi_n \sum_{ t =  -  \lfloor \gep_n T_n^*\rfloor }^{ \lfloor \gep_n T_n^* \rfloor}  e^{- \frac{\varsigma_n(t)}{2 a_n^2} }  \,,
  \end{split} \]
which clearly gives that $Z_{n,h_n} \sim \frac{2}{\pi} \bar\psi_n T_n^* \theta_n(a)$, because we already know that taking $\gep_n$ going to zero sufficienly slowly we have $\lim_{n\to\infty}\PP_{n,h_n}( |T_n-T_n^*| >\gep_n T_n)=0$.

Moreover, applying again Lemma~\ref{lem:local0}, we also get that for any fixed integer $t\in \ZZ$,
\[
 Z_{n,h_n}\Big(   T_n = \lfloor T_n^*-2\rfloor +t \Big) 
  \sim \frac{2}{\pi} T_n^* \bar \psi_n 
e^{- \frac{\varsigma_n(t)}{2 a_n^2} } \,,
\]
using the same calculation as above.
This concludes the proof of Proposition~\ref{prop-1/4}.
\end{proof}

\section{Range endpoints and gambler's ruin estimates}
\label{sec-ruine}

\subsection{Gambler's ruin estimates}

We consider a band $[0,T]$ with $T$ some positive integer, and choose a starting point $0 \leq z \leq T$. We denote $\tau_0\defeq \min\{n\geq 0 \,, S_n =0\}$, resp. $\tau_T\defeq\min\{n\geq 0 \,, S_n =T\}$,  the hitting time of the edge at $0$, resp. at $T$. We also denote $\tau \defeq \tau_0 \wedge \tau_T$. We recall the formulae of \cite[\S 14.5]{Feller} for the ruin problem, in the case of a symmetric walk.
We use the notation $n\leftrightarrow z$ if $n-z$ is even and denote by $\PP_z$ the law of the simple random walk starting at $z \in \ZZ$.

\begin{proposal}
For any $z \in \llbracket 1,T-1 \rrbracket$ and  $n > 1$,
\begin{equation}\label{Feller biais 0}
    \probaM{z}{\tau = \tau_0 = n} = \frac{2}{T} \sum_{1 \leq k < T/2} \cos^{n-1} \left( \frac{\pi k}{T} \right) \sin \left( \frac{\pi kz}{T} \right) \sin \left( \frac{\pi k}{T} \right) \indic{n \leftrightarrow z} \, .
\end{equation}
By symmetry, we also have $\probaM{z}{\tau = \tau_T = n} =
\probaM{T-z}{\tau = \tau_0 = n}$:
\begin{equation}\label{Feller biais T}
    \probaM{z}{\tau = \tau_T = n} = \frac{2}{T} \sum_{1 \leq k < T/2} (-1)^{k+1} \cos^{n-1} \left( \frac{\pi k}{T} \right) \sin \left( \frac{\pi kz}{T} \right) \sin \left( \frac{\pi k}{T} \right) \indic{n \leftrightarrow T-z} \,.
\end{equation}
\end{proposal}

\noindent
Note that if $T-z \leftrightarrow n \leftrightarrow z$, we have
\begin{equation*}
    \probaM{z}{\tau = n} = \frac{4}{T} \sum_{1 \leq k < T/4} \cos^{n-1} \left( \frac{ (2k - 1)\pi}{T} \right) \sin \left( \frac{ (2k - 1)\pi z}{T} \right) \sin \left( \frac{(2k - 1)\pi }{T} \right).
\end{equation*}

Let us now give the \textit{sharp} asymptotic behavior of the probabilities~\eqref{Feller biais 0}-\eqref{Feller biais T} above.
Recall the definition \eqref{g DL}: $g(T) = -\log \cos(\frac\pi T)$.
By symmetry, we only deal with the case $z\in \llbracket 0,\frac T2 \rrbracket$.
\begin{lemma}
\label{lemme-ruine}
Suppose that $T = T(n) \to\infty$ as $n \to \infty$ and that $\lim_{n\to\infty} \frac{n}{T^2} = +\infty$.
Then, we have the following asymptotics:
for all  $z\in \llbracket 0,\frac T2 \rrbracket$,
\begin{equation}\label{eq-proba-approx1}
\probaM{z}{\tau = \tau_0 = n} = \big(1 +\grdO (e^{-\frac{\pi^2 n}{T^2}})  \big) \frac{2}{T} \sin \left( \frac{z\pi}{T} \right) \tan \left( \frac{\pi}{T} \right) e^{-g(T) n} \indic{n \leftrightarrow z} \,,
\end{equation}
\begin{equation}
\probaM{z}{\tau = \tau_T = n} = \big(1 +\grdO  (e^{-\frac{\pi^2 n}{T^2}}) \big) \frac{2}{T} \sin \left( \frac{z\pi}{T} \right) \tan \left( \frac{\pi}{T} \right) e^{-g(T) n} \indic{n \leftrightarrow T-z} \,.
\end{equation}
Here,  $\grdO  (e^{-\frac{\pi^2 n}{T^2}})$ is uniform in
$z$.
\end{lemma}

\begin{remark}
We recall that equations such as \eqref{eq-proba-approx1} are to be understood in the sense that
\[ \exists C > 0, \forall T = T(n), \quad \sup_{\substack{z \in \llbracket 0,\frac T2 \rrbracket\\ n \leftrightarrow z}} \left|  \frac{\probaM{z}{\tau = \tau_T = n}}{\frac{2}{T} \sin \left( \frac{z\pi}{T} \right) \tan \left( \frac{\pi}{T} \right) e^{-g(T) n}} - 1 \right| \leq C e^{-\frac{\pi^2 n}{T^2}} \quad \text{as } n \to \infty. \]
\end{remark}

\begin{proof}
The proof is inspired by \cite[Appendix B]{caravenna2009depinning}, but we need here a slightly sharper version.
In \eqref{Feller biais 0} and \eqref{Feller biais T} we denote 
$V_0 = V_0(n,T)$ the first term:
\[ 
V_0 = \frac{2}{T}\cos^{n-1} \left( \frac{\pi}{T} \right) \sin \left( \frac{\pi z}{T} \right) \sin \left( \frac{\pi}{T} \right) = \frac{2}{T} \sin \left( \frac{z\pi}{T} \right) \tan \left( \frac{\pi}{T} \right) e^{-g(T) n}  \,.
\]
It remains to control the remaining terms.
We let
\[
V_1  \defeq  \frac{2}{T}  \sum_{2\leq k <  T/2} \cos^{n-1} \left( \frac{\pi k}{T} \right) \sin \left( \frac{\pi kz}{T} \right) \sin \left( \frac{\pi k}{T} \right) \,, 
\]
and we only need to bound $V_1/V_0$.
Using the bounds $\frac{2}{\pi}x \leq \sin(x) \leq x$ for $x \in [0,\frac\pi2]$, we get that
\[
\frac{V_1}{V_0} \leq  \frac{\pi^2}{4}\sum_{2\leq k <  T/2}   k^2 \bigg(\frac{\cos \left( \frac{\pi k}{T} \right)}{\cos\left(\frac{\pi}{T}\right)} \bigg)^{n-1}  \,.
\]
Now, as $\frac{k}{T}\to 0$, we have
\[
\frac{\cos \left( \frac{\pi k}{T} \right)}{\cos\left(\frac{\pi}{T}\right)} 
= 1- \frac{\pi^2 (k^2-1)}{2T^2}  \Big(1+\grdO\Big( \frac{k^2}{T^2} \Big) \Big) \,.
\]
Hence, the l.h.s.\ is bounded by $\exp( - \frac{2\pi^2 (k^2-1)}{5T^2})$ provided that $\frac{k}{T}\leq \varepsilon$, for some given $\varepsilon \in (0,\frac12)$.
If $\frac{k}{T}\geq \varepsilon$,
we can simply bound
$\cos \left( \frac{\pi k}{T} \right) \leq \cos(\pi \varepsilon) \leq e^{- \frac12 \pi^2 \varepsilon^2}$.
We therefore get that $V_1/V_0$ is bounded by a constant times
\[
\sum_{2\leq k <  \varepsilon T}  k^2 e^{- \frac{2 n \pi^2 (k^2-1)}{5T^2}}
+ \sum_{\varepsilon T\leq k <  T/2}  k^2 e^{-\frac12  n \pi^2 \varepsilon^2 }
 \,.
\]
\par For the first sum, we write that it is 
\[
\begin{split}
e^{\frac{2\pi^2n }{5T^2}}  T^3  \times  \frac{1}{T} \sum_{2\leq k <  \varepsilon T}  \frac{k^2}{T^2} e^{- \frac{2 \pi^2 n k^2}{5T^2}}
& \leq e^{\frac{2\pi^2n}{5T^2}}  T^3 \int_{2/T}^{\infty} x^2 e^{-\frac{2\pi^2}{5} n x^2} \dd x  \\
& = e^{\frac{2\pi^2n}{5T^2}}   \frac{T^3}{n^{3/2}} \int_{2\sqrt{n}/T}^{\infty} u^2 e^{- \frac{2\pi^2}{5} u^2} \dd x 
\leq C \frac{T^2}{n}  \exp\Big( -\frac{7\pi^2 n}{5 T^2} \Big) \,,
\end{split}
\]
using that $\int_{v}^{\infty} u^2 e^{- \frac{2\pi^2}{5} u^2} \mathrm{d} u  \sim (cst.)\, v\, e^{- \frac{2\pi^2}{5} v^2}$ as $v\to\infty$.
This term is therefore bounded by a constant times $\exp( -\frac{\pi^2 n}{T^2})$, as $n/T^2$ goes to infinity.

For the other sum, we bound it by a constant times 
\[
T^3 \exp\Big(-\frac12  n \pi^2 \varepsilon^2\Big) 
\leq n^{3/2} \exp\Big(-\frac12  n \pi^2 \varepsilon^2 \Big) = \bar{o}\Big(\exp\Big( -\frac{\pi^2 n}{T^2} \Big) \Big) \,,
\]
as $\frac{n}{T^2} \to +\infty$ and $T\to\infty$.
We have therefore shown that $V_1/V_0$ is bounded by
a constant times $\exp( -\frac{\pi^2 n}{T^2})$, which concludes the proof.
\end{proof}


We now obtain an expression for the probability of staying in the band $[0,T]$, without touching the border, during a time $n \gg T^2$. 

\begin{lemma}\label{lemme conf strict} 
Assume that $T = T(n)\to\infty$ and that $\lim_{n\to\infty} \frac{n}{T^2} =+\infty$. Then, we have:
\begin{itemize}
\item If $T$ is odd,
 \begin{equation}
\probaM{z}{\tau > n} = \frac{2}{T} \sin \left( \frac{z\pi}{T} \right) \frac{1}{\tan \left( \frac{\pi}{2T} \right)} e^{-n g(T)} \big(1 +\grdO (e^{-\frac{\pi^2 n}{T^2}})  \big)  \,.
\end{equation}

\item If $T$ is even, letting $a = \indic{n \leftrightarrow z}$
\begin{equation}
\probaM{z}{\tau > n} = \frac{4}{T} \sin \left( \frac{z\pi}{T} \right) \frac{\cos^a\left( \frac{\pi}{T} \right)}{\sin \left( \frac{\pi}{T} \right)} e^{-n g(T)} \big(1 +\grdO (e^{-\frac{\pi^2 n}{T^2}})  \big)  \,.
\end{equation}
\end{itemize}
In particular, with a Taylor expansion, 
we get that
\begin{equation}
\label{equiv-pastouche}
f_n(z,T)\defeq\probaM{z}{\tau > n} = \frac{4}{\pi} \sin \left( \frac{z\pi}{T} \right)  e^{-n g(T)}  \left[1 +\grdO(T^{-2}) + \grdO\big( e^{-\frac{\pi^2 n}{T^2}} \big)  \right]\,,
\end{equation}
and note that if $n \geq \frac14 T^2 \log T$
then $e^{-\frac{\pi^2 n}{T^2}} \leq T^{-\frac14 \pi^2} \leq T^{-2}$.
\end{lemma}

%

\begin{proof}
First of all, we write
\begin{equation}
\label{eq:decomposegambler}
\probaM{z}{\tau > n} = \sum_{k>n} \Big(\probaM{z}{\tau = \tau_0 = k} + \probaM{z}{\tau = \tau_T = k} \Big) \,.
\end{equation}

When $T$ is odd, then in~\eqref{eq:decomposegambler}, for each $k$ in the sum there is only one term which is non-zero:
applying Lemma~\ref{lemme-ruine}
to estimate that term, we get 
\[
\begin{split}
\probaM{z}{\tau > n}  
 & = \big(1 +\grdO (e^{-\frac{\pi^2 n}{T^2}})  \big)
  \frac{2}{T} \sin \left( \frac{z\pi}{T} \right) \tan \left( \frac{\pi}{T} \right)  \sum_{k>n} e^{-g(T) k}  \\
  & =  \big(1 +\grdO (e^{-\frac{\pi^2 n}{T^2}})  \big)
  \frac{2}{T} \sin \left( \frac{z\pi}{T} \right) \tan \left( \frac{\pi}{T} \right) e^{-g(T)n} \frac{\cos\left( \frac{\pi}{T} \right)  }{1- \cos\left( \frac{\pi}{T} \right) }  \,,
 \end{split}
\]
recalling that $e^{-g(T)} = \cos\left( \frac{\pi}{T} \right)$.
This gives the desired result since $\frac{\sin(\theta)}{1-\cos(\theta)} = \frac{1}{\tan(\theta/2)}$.

\smallskip
When $T$ is even, notice that in~\eqref{eq:decomposegambler},
either $k\leftrightarrow z$ and then both terms are non-zero
or $k\not\leftrightarrow z$ and then both terms are zero.
Applying Lemma~\ref{lemme-ruine}, we get 
\[
\probaM{z}{\tau > n} = \big(1 +\grdO (e^{-\frac{\pi^2 n}{T^2}})  \big)  \frac{4}{T} \sin \left( \frac{z\pi}{T} \right)   \tan \left(\frac{\pi}{T} \right) \sum_{k > n} e^{-g(T) k} \indic{k \leftrightarrow z} \,.
\]
\par To deal with the last sum,
denote $n^*= n^*(z) \defeq \min\{ k>n, k\leftrightarrow z\}$: 
note that $n^*$ is equal to $n +1+ a$ with $a = \indic{n\leftrightarrow z}$.
The indices for which the term is not zero can be written as $k = n^* + 2j$ and thus
\[
\sum_{k > n} e^{-g(T) k} \indic{k \leftrightarrow z}
= e^{-n^* g(T)}  \sum_{j \geq 0} e^{-2 g(T) j}  
= e^{-n g(T)} \frac{\cos^{1+a}\left(\frac{\pi}{T} \right) }{1 -\cos^{2}\left(\frac{\pi}{T} \right) } \,,
\]
recalling that $e^{-g(T)} = \cos\left( \frac{\pi}{T} \right)$.
This gives the announced expression.
\end{proof}

\subsection{Range estimates}

Recall the definition of the event $E_x^y(n) = \{M_n^{-} =- x, M_n^+ =y\}$ for any two positive integers $x$ and $y$ (the case where one equals $0$ is dealt in Section \ref{Exy positive RW}).
We use Lemma~\ref{lemme conf strict} 
to estimate $\PP_0(E_x^y(n))$, \textit{i.e.}\ to prove Theorem~\ref{prop equiv proba}.
From this point onward, we always denote $T\defeq x+y$.
Using the spatial invariance of the random walk, we study the probability starting from $x$ to stay in the strip $[0,T]$ \textit{and} to touch both borders before time~$n$. The symmetry of the walk allows us to assume $x \leq y$ and $0 \leq x \leq \frac{T}{2}$.
\par We now write the probability of $E_x^y(n)$ as the following differences
\[ \begin{split}
    \probaM{0}{E_x^y(n)} &= \probaM{0}{M_n^+ < y + 1, M_n^- = -x} - \probaM{0}{M_n^+ < y, M_n^- = -x}\\
    &= \probaM{0}{M_n^+ < y + 1, M_n^- > -x-1} - \probaM{0}{M_n^+ < y + 1, M_n^- > -x} \\
    & \qquad  - \big[ \probaM{0}{M_n^+ < y, M_n^- > -x-1} - \probaM{0}{M_n^+ < y, M_n^- > -x} \big] \,.
\end{split}  
 \]
 Then each of those probabilities is of a strict confinement event with different strips widths and starting points:
 we get
\begin{equation}
\label{Exy:difference}
 \probaM{0}{E_x^y(n)} 
 = f_n(x+1,T+2) - f_n(x,T+1)-f_n(x+1,T+1) + f_n(x,T) \,.
\end{equation}
We can
therefore use Lemma~\ref{lemme conf strict} 
to estimate each of these terms.
We will find different asymptotics
depending on whether $\frac{n}{T^3}$ goes to $0$ or to $+\infty$
(or converges to a constant).
This ratio is known to be the relevant quantity when studying such constrained random walk (see \cite{caravenna2009depinning}).
The main reason is that
in the expansion~\eqref{equiv-pastouche}
we have that $e^{-n g(T)} \sim e^{-n g(T+1)} \sim e^{-n g(T+2)}$
if and only if $\frac{n}{T^3}$ goes to $0$ (see below):
this case will prove to be more intricate because of several cancellations
in~\eqref{Exy:difference}.

\subsubsection{First case: \texorpdfstring{$\lim_{n\to\infty}\frac{n}{T^3} =+\infty$}{}}

In that case, recalling that $e^{-g(T)} = \cos(\frac{\pi}{T})$ we have 
\[
\frac{e^{-g(T+1) n}}{e^{-g(T+2)n}}  =\Big(1 - (1+\bar{o}(1)) \frac{\pi^2}{T^3} \Big)^n \xrightarrow{n\to\infty} 0 \,,
\quad
\frac{e^{-g(T) n}}{e^{-g(T+2)n}} =\Big(1 - (1+\bar{o}(1)) \frac{2\pi^2}{T^3} \Big)^n \xrightarrow{n\to\infty} 0 \,.
\]
Therefore,
in view of~\eqref{equiv-pastouche},
we have that $f_n(x,T+1),f_n(x+1,T+1),f_n(x,T)$
are all negligible compared to $f_n(x+1,T+2)$.
Using \eqref{Exy:difference},
and~\eqref{equiv-pastouche},
we therefore get that
\begin{equation}\label{equiv infini}
\probaM{0}{E_x^y(n)}  =(1+ \bar{o}(1) )  \frac{4}{\pi} \sin \left(\frac{ \pi(x+1)}{T} \right)  e^{-g(T+2)n} \,,
\end{equation}
where the $\bar{o}(1)$ depends only on $T$ and is uniform in $x$.
Using also that $\sin(\frac{\pi(x+1)}{T+2}) = (1+\grdO(\frac1T)) \sin(\frac{\pi (x+1)}{T})$ uniformly in $x$,
we get the desired result.

\subsubsection{Second case: \texorpdfstring{$\lim_{n\to\infty}\frac{n}{T^3} =\alpha \in(0,+\infty)$}{}}

Similarly
as above, we have 
\[
\lim_{n\to\infty} \frac{e^{-g(T+2) n}}{e^{-g(T)n}}  =e^{2\alpha \pi^2} \,,
\quad
\lim_{n\to\infty}  \frac{e^{-g(T+1) n}}{e^{-g(T)n}} = e^{\alpha \pi^2} \,.
\]
Therefore, using~\eqref{Exy:difference} and~\eqref{equiv-pastouche},
we get that
\[
\begin{split}
\frac{\pi}{4} &e^{g(T)n} \probaM{0}{E_x^y(n)} \\
&= (1+\bar{o}(1))  \left[ \sin\left( \frac{\pi (x+1)}{T}\right)e^{2\alpha \pi^2} + \sin\left( \frac{\pi x}{T}\right) -
\left[\sin\left( \frac{\pi (x+1)}{T}\right)+\sin\left( \frac{\pi x}{T}\right) \right] e^{\alpha \pi^2} \right]
\end{split}
\]
where we have used that 
$\sin(\frac{\pi (x+1) }{T+2}) =(1+\bar{o}(1)) \sin(\frac{\pi (x+1)}{T})$
with $\bar{o}(1)$ uniform in $x$,
and similarly with $\sin(\frac{\pi (x+1) }{T+1})$,
$\sin(\frac{\pi x }{T+1})$.
This gives the announced asymptotics.

Let us stress that, in the case where
$x=\grdO(1)$,
we get
\begin{equation}
  \label{equiv proba alpha0}
\probaM{0}{E_x^y(n)} =(1+\bar{o}(1))  \frac{4}{T} (e^{\alpha \pi^2}-1) (e^{\alpha \pi^2} (x+1)-x) e^{-g(T)n}.
\end{equation}
If on the other hand we have $x\to\infty$, then
we have
\begin{equation}
  \label{equiv proba alpha}
\probaM{0}{E_x^y(n)} =(1+\bar{o}(1))  \frac{4}{\pi}  (e^{\alpha \pi^2}-1)^2 \sin \left( \frac{x\pi}{T} \right) e^{-g(T)n} \,,
\end{equation}
and one can also write $(e^{\alpha \pi^2}-1)^2 = 4 e^{\alpha \pi^2} \cosh^2(\alpha \pi^2)$.

\subsubsection{Last case: \texorpdfstring{$\lim_{n\to\infty} \frac{n}{T^3}=0$}{}}

When $\lim_{n\to\infty} \frac{n}{T^3}=0$, the first order terms all cancel each other and we have to look further in the asymptotic expansion.
We will show that
\[
\frac{\pi}{4} \probaM{0}{E_x^y(n)} e^{g(T+1)n} 
= (1+\bar{o}(1))  \frac{\pi^4n^2}{T^6} \sin\Big( \frac{\pi (x+\frac12)}{T}  \Big) 
+ (1+\bar{o}(1))  \frac{\pi^3 n}{T^4}  \left(1-\frac{2x}{T} \right)\, \cos\Big( \frac{\pi x}{T}  \Big)  \,.
\]
This proves the desired result
since the second term can only be dominant if $x/T \to 0$;
it is actually dominant in particular if $x=\grdO(1)$.


Recalling that $\lim_{n\to\infty} \frac{n}{T^2} = +\infty$ (in particular $\frac{n}{T^4} = \bar{o}(\frac{n^2}{T^6})$), we have the following expansions:
\[
\frac{e^{-g(T) n}}{e^{-g(T+1)n}}  = 1 -\frac{\pi^2 n}{T^3} +(1+\bar{o}(1)) \frac{ \pi^4 n^2}{2T^6}  \,,
\qquad
\frac{e^{-g(T+2) n}}{e^{-g(T+1)n}}  = 1 +\frac{\pi^2 n}{T^3} +(1+\bar{o}(1))\frac{ \pi^4 n^2}{2T^6} (1+\bar{o}(1)) \,.
\]
Hence, from~\eqref{Exy:difference}, using~\eqref{equiv-pastouche}
(with the fact that $n\geq \frac14 T^2\log T$ so that $\grdO\big( e^{-\frac{\pi^2 n}{T^2}} \big) =\bar{o}(T^{-2})$)
we get
\begin{equation*}\label{proba cas 0} 
\begin{split}
   \frac{\pi}{4}  \probaM{0}{E_x^y(n)} e^{g(T+1)n}   =   
       \Bigg\{ \sin  \left( \frac{\pi x}{T} \right) & \Big[  1 -\frac{\pi^2 n}{T^3} +(1+\bar{o}(1)) \frac{ \pi^4 n^2}{2T^6} \Big]    \\
      - & \left( \sin \left( \frac{\pi x}{T+1} \right)  + \sin \left( \frac{\pi (x+1)}{T+1} \right) \right)
  \Big[ 1 + \grdO(T^{-2})\Big]  \\
  & \qquad    + \sin \Big( \frac{\pi(x+1)}{T+2} \Big)\Big[ 1 +\frac{\pi^2 n}{T^3} +(1+\bar{o}(1)) \frac{ \pi^4 n^2}{2T^6}  \Big]  \Bigg\}\, .
\end{split} 
\end{equation*}
Note that we absorbed all terms $\grdO(T^{-2})$ in the $\bar{o}(\frac{n^2}{T^6})$,
since $\lim_{n\to\infty}\frac{n}{T^2} = +\infty$.
Hence, 
we get that
\[
\frac{\pi}{4}  \probaM{0}{E_x^y(n)} e^{g(T+1)n}   =
 (1+\bar{o}(1)) \frac{ \pi^4 n^2}{2T^6} \Big( \sin  \left( \frac{\pi x}{T} \right) + \sin \Big( \frac{\pi(x+1)}{T} \Big)\Big) 
+ A+B \,,
\]
with
\[
\begin{split}
    A &= \sin \left( \frac{x\pi}{T} \right) -  2  \sin \Big( \frac{\pi (x+\frac12)}{T+1} \Big)  + \sin \Big( \frac{\pi(x+1)}{T+2} \Big),\\
    B &=  \frac{n\pi^2}{ T^3}  \left[  \sin \Big( \frac{\pi(x+1)}{T+2} \Big) -  \sin \Big( \frac{\pi x}{T} \Big) \right] \,.
\end{split}
\]
Here, for $A$, we have also used that $\sin ( \frac{\pi x}{T+1} )  + \sin ( \frac{\pi (x+1)}{T+1} ) = 2 \sin ( \frac{\pi (x+\frac12)}{T+1} )  \cos ( \frac{\pi}{2(T+1)} )$
with $\cos ( \frac{\pi}{2(T+1)} ) = 1+\grdO(T^{-2})$ 
and absorbed the $\grdO(T^{-2})$ in the $\bar{o}(\frac{n^2}{T^6})$.
We show below that $A$ is negligible compared to $B$,
so let us start by estimating $B$.

\paragraph*{Term B.}
Note that  setting $v\defeq \frac{T}{2}-x$ we have
\[
B= \frac{n\pi^2}{T^3} \left[ \cos \Big( \frac{\pi v}{ T+2} \Big) -   \cos \Big( \frac{\pi v}{ T} \Big)  \right] \,.
\]
Using the formula for the difference of cosines,
we get that
\[
\begin{split}
\cos \Big( \frac{\pi v}{ T+2} \Big) -   \cos \Big( \frac{\pi v}{ T+1} \Big) 
& = 2 \sin \Big( \frac{\pi v}{ 2(T+1)(T+2)} \Big)  \sin\Big( \frac{\pi v}{T+1}  \frac{T+\frac32}{T+2}  \Big) \\
& = (1+\bar{o}(1)) \frac{\pi v}{T^2} \, \sin\Big( \frac{\pi v}{T}  \Big) \,.
\end{split}
\]
We end up with
\[
B=(1+\bar{o}(1))\frac{n\pi^3 v}{T^5} \, \sin\Big( \frac{\pi v}{T}  \Big) \,.
\]

\paragraph*{Term A.}
As far as $A$ is concerned,
notice that setting $v\defeq \frac{T}{2}-x$ we have
\[
A = \cos \Big( \frac{\pi v}{T}  \Big) -
2  \cos \Big( \frac{\pi v}{T+1}  \Big) +
\cos \Big( \frac{\pi v}{T+2} \Big) \,.
\]
Using the formula for the 
difference of cosines, we get that $A/2$ is equal to
\[
\begin{split}
 - \sin & \left( \frac{\pi v}{2T(T+1)}\right)  \sin \left( \frac{\pi v}{T} \frac{T+\frac12}{T+1}\right) 
+  \sin \left( \frac{\pi v}{2(T+1)(T+2)}\right)  \sin \left( \frac{\pi v}{T} \frac{T(T+\frac32)}{(T+1)(T+2)}\right)   \\
& =\frac{\pi v}{2T^2}  \Bigg[ \sin \left( \frac{\pi v}{T} \frac{T(T+\frac32)}{(T+1)(T+2)}\right)- \sin \left( \frac{\pi v}{T} \frac{T+\frac12}{T+1}\right)  \Bigg]
 + \grdO\Big( \frac{v}{T^3}  \sin \left( \frac{\pi v}{T} \right)\Big) \,,
\end{split}
\]
where we have used that $\sin ( \frac{\pi v}{2T(T+1)}) =  \frac{\pi v}{2T^2} (1+\grdO(T^{-1}))$ and similarly
for $\sin ( \frac{\pi v}{2(T+1)(T+2)}) $.
Using the formula for the difference 
of sines, we get 
\[
\begin{split}
\sin \left( \frac{\pi v}{T} \frac{T(T+\frac32)}{(T+1)(T+2)}\right) & - \sin \left( \frac{\pi v}{T} \frac{T+\frac12}{T+1}\right)  \\
& = -2 \sin \left( \frac{\pi v}{2T (T+2)}\right) 
\cos\left( \frac{\pi v}{T}  \left[ 1+\grdO(T^{-1}) \right]  \right) \,.
\end{split}
\]
Hence, 
we end up with $A = \grdO(\frac{v^2}{T^4}\cos(\frac{\pi v}{T})) +  \grdO(\frac{v}{T^3} \sin(\frac{\pi v}{T}))$.
This is negligible compared to $B$,
since $\lim_{n\to\infty} \frac{n}{T^2} =+\infty$.

\subsubsection{Estimates for a positive random walk}\label{Exy positive RW}

Note that our estimates need adjustments when $x = 0$ or $x = T$. We go back to correct \eqref{Exy:difference} in order to take into account that $0$ is the starting point of the walk. We write
\[ 
\begin{split}
\probaM{0}{E_0^y(n)} & = \probaM{0}{M_n^->-1, M_n^+ < y+1} - \probaM{0}{M_n^->-1, M_n^+ < y} \\
& = f_n(1,T+2) - f_n(1,T+1)
\end{split}
\]
Note that $y=T$ but we will keep separating the notations $y$ and $T$. thus, we have
\begin{equation}
    \probaM{0}{E_0^y(n)} = \frac{4}{\pi} \left[ \sin \left( \frac{\pi}{T+2} \right) e^{-g(T+2)n} - \sin \left( \frac{\pi}{T+1} \right) e^{-g(T+1)n} + \grdO(T^{-3}) \right]
\end{equation}
Once again we have different asymptotics depending on the ratio $n/T^3$ that we rapidly present in the following

\paragraph*{Case $\frac{n}{T^3} \to +\infty$}

As previously, $e^{-g(T+2)n}$ is the dominant term and thus
\begin{equation}
    \probaM{0}{E_0^y(n)} = \frac{4}{\pi}(1+\bar{o}(1)) \sin \left( \frac{\pi}{T+2} \right) e^{-g(T+2)n},
\end{equation}
and we get the formula of \eqref{equiv infini} applied to $x = 0$.

\paragraph*{Case $\frac{n}{T^3} \to \alpha$}

Factorize by $e^{-g(T)n}$ as in the general case, we thus write
\[ \probaM{0}{E_0^y(n)} = \frac{4}{\pi} \left[ \sin \left( \frac{\pi}{T+2} \right) e^{2\alpha \pi^2} - \sin \left( \frac{\pi}{T+1} \right) e^{\alpha \pi^2} + \grdO(T^{-3}) \right] e^{-g(T)n}. \]
That can we rewritten as
\[
\probaM{0}{E_0^y(n)} = \frac{4}{\pi} (1+\bar{o}(1)) e^{\alpha \pi^2} (e^{\alpha \pi^2} - 1) \sin \left( \frac{\pi}{T} \right) e^{-g(T)n},
\]
which is exactly \eqref{equiv proba alpha0} taken at $x=0$.

\paragraph*{Case $\frac{n}{T^3} \to 0$}

In this case, we again factorize by $e^{-g(T+1)n}$ and write
\[ 
\begin{split}
\probaM{0}{E_0^y(n)} = \frac{4}{\pi} \bigg[ \sin \left( \frac{\pi}{T+2} \right) \Big[ 1 & + \frac{n\pi^2}{T^3} + \frac{n^2 \pi^4}{2T^6}(1+\bar{o}(1)) \Big] \\
& - \sin \left( \frac{\pi}{T+1} \right) + \grdO(T^{-3}) \bigg] e^{-g(T+1)n}. 
\end{split}
\]
We are left to compare all the terms in this expression :
\[ A = \sin \left( \frac{\pi}{T+2} \right) - \sin \left( \frac{\pi}{T+1} \right) = -2 \sin \left( \frac{\pi}{2(T+1)(T+2)} \right) \cos \left( \frac{\pi}{2T} \right) \sim -\frac{\pi}{T^2}, \]
\[ B = \frac{n \pi^2}{T^3} \sin \left( \frac{\pi}{T+2} \right) \sim \frac{n\pi^3}{T^4}, \qquad D = \frac{n^2 \pi^4}{T^6} \sin \left( \frac{\pi}{T+2} \right) \sim \frac{n^2\pi^5}{T^7}. \]
See that $A \ll B$ and $D \ll B$ using both $\frac{n}{T^2} \to \infty$ and $\frac{n}{T^3} \to 0$, meaning that
\[ \probaM{0}{E_0^y(n)} = \frac{4n\pi}{T^3} (1+\bar{o}(1)) \sin \left( \frac{\pi}{T+2} \right) e^{-g(T+1)n}. \]

\subsection*{Acknowledgements}

The author would like to thank his PhD advisors Quentin Berger and Julien Poisat for reviewing this paper and for their continual help.

\printbibliography[heading=bibintoc]

\end{document}